\documentclass[a4paper,11pt,reqno]{amsart}
\usepackage{latexsym, amsmath, amsfonts, amssymb, amsthm, amscd, epsfig}
\usepackage{url}
\usepackage{dsfont}
\usepackage[font=small, labelfont={sc}]{caption}
\usepackage{subfig}

\usepackage[a4paper,scale={0.73,0.75},marginratio={1:1},footskip=10mm,headsep=10mm]{geometry}

\usepackage{color}

\usepackage{hyperref}

\numberwithin{equation}{section}

\def \beq {\begin{eqnarray}}
\def \eeq {\end{eqnarray}}
\def \beqn {\begin{eqnarray*}}
\def \eeqn {\end{eqnarray*}}

\newtheorem{theorem}{Theorem}
\newtheorem{itlemma}[theorem]{Lemma}
\newtheorem{itproposition}[theorem]{Proposition}
\newtheorem{itcorollary}[theorem]{Corollary}
\newtheorem{itremark}[theorem]{Remark}
\newtheorem{itdefinition}[theorem]{Definition}
\newtheorem{itexample}[theorem]{Example}
\newtheorem{itclaim}[theorem]{Claim}
\newtheorem{itfact}[theorem]{Fact}

\newtheorem{conj}{Conjecture}[section]

\newenvironment{fact}{\begin{itfact}\rm}{\end{itfact}}
\newenvironment{claim}{\begin{itclaim}\rm}{\end{itclaim}}
\newenvironment{lemma}{\begin{itlemma}}{\end{itlemma}}
\newenvironment{remark}{\begin{itremark}\rm}{\end{itremark}}
\newenvironment{corollary}{\begin{itcorollary}}{\end{itcorollary}}
\newenvironment{proposition}{\begin{itproposition}}{\end{itproposition}}
\newenvironment{definition}{\begin{itdefinition}\rm}{\end{itdefinition}}
\newenvironment{example}{\begin{itexample}\rm}{\end{itexample}}

\newcommand{\be}[1]{\begin{equation}\label{#1}}
\newcommand{\ee}{\end{equation}}
\newcommand{\bl}[1]{\begin{lemma}\label{#1}}
\newcommand{\br}[1]{\begin{remark}\label{#1}}
\newcommand{\brs}[1]{\begin{remarks}\label{#1}}
\newcommand{\bt}[1]{\begin{theorem}\label{#1}}
\newcommand{\bd}[1]{\begin{definition}\label{#1}}
\newcommand{\bp}[1]{\begin{proposition}\label{#1}}
\newcommand{\bc}[1]{\begin{corollary}\label{#1}}
\newcommand{\bfact}[1]{\begin{fact}\label{#1}.}
\newcommand{\bex}[1]{\begin{example}\label{#1}.}
\newcommand{\ec}{\end{corollary}}
\newcommand{\efact}{\end{fact}}
\newcommand{\eex}{\end{example}}
\newcommand{\el}{\end{lemma}}
\newcommand{\er}{\end{remark}}
\newcommand{\ers}{\end{remarks}}
\newcommand{\et}{\end{theorem}}
\newcommand{\ed}{\end{definition}}
\newcommand{\ep}{\end{proposition}}
\newcommand{\epr}{\end{proof}}
\newcommand{\bpr}{\begin{proof}}
\newcommand{\bcl}[1]{\begin{claim}\label{#1}}
\newcommand{\ecl}{\end{claim}}

\newcommand{\ecs}{\end{corollary}}
\newcommand{\eers}{\end{exercise}}
\newcommand{\eexs}{\end{example}}
\newcommand{\eems}{\end{example}}
\newcommand{\els}{\end{lemma}}
\newcommand{\eles}{\end{lemmaex}}
\newcommand{\ets}{\end{theorem}}
\newcommand{\eds}{\end{definition}}
\newcommand{\eps}{\end{proposition}}
\newcommand{\bi}{\begin{itemize}}
\newcommand{\ei}{\end{itemize}}
\newcommand{\ben}{\begin{enumerate}}
\newcommand{\een}{\end{enumerate}}

\def\vbar{\mathchoice{\vrule height6.3ptdepth-.5ptwidth.8pt\kern-.8pt}
   {\vrule height6.3ptdepth-.5ptwidth.8pt\kern-.8pt}
   {\vrule height4.1ptdepth-.35ptwidth.6pt\kern-.6pt}
   {\vrule height3.1ptdepth-.25ptwidth.5pt\kern-.5pt}}
\def\fudge{\mathchoice{}{}{\mkern.5mu}{\mkern.8mu}}
\def\bbc#1#2{{\rm \mkern#2mu\vbar\mkern-#2mu#1}}
\def\bbb#1{{\rm I\mkern-3.5mu #1}}
\def\bba#1#2{{\rm #1\mkern-#2mu\fudge #1}}
\def\bb#1{{\count4=`#1 \advance\count4by-64 \ifcase\count4\or\bba A{11.5}\or
   \bbb B\or\bbc C{5}\or\bbb D\or\bbb E\or\bbb F \or\bbc G{5}\or\bbb H\or
   \bbb I\or\bbc J{3}\or\bbb K\or\bbb L \or\bbb M\or\bbb N\or\bbc O{5} \or
   \bbb P\or\bbc Q{5}\or\bbb R\or\bbc S{4.2}\or\bba T{10.5}\or\bbc U{5}\or
   \bba V{12}\or\bba W{16.5}\or\bba X{11}\or\bba Y{11.7}\or\bba Z{7.5}\fi}}

\def \N {{\mathbb N}}

\def \ra {\rightarrow }

\newcommand{\ba}[1]{\addtocounter{for}{1} \begin{eqnarray}\label{#1}}
\newcommand{\ea}{\end{eqnarray}}

\def\sqr#1#2{{\vcenter{\vbox{\hrule height .#2pt
                             \hbox{\vrule width .#2pt height#1pt \kern#1pt
                                   \vrule width .#2pt}
                             \hrule height .#2pt}}}}

\def\pmb#1{\setbox0=\hbox{#1}%
   \kern-.025em\copy0\kern-\wd0
   \kern.05em\copy0\kern-\wd0
   \kern-.025em\raise.0433em\box0 }
\def\sqr#1#2{{\vcenter{\vbox{\hrule height.#2pt
     \hbox{\vrule width.#2pt height#1pt \kern#1pt
   \vrule width.#2pt}\hrule height.#2pt}}}}

\def\d{\delta}
\def\l{\lambda}

\DeclareMathOperator{\1}{\mathds{1}}
\DeclareMathOperator{\Law}{Law}

\definecolor{dblue}{RGB}{25,0,102}
\definecolor{dgreen}{rgb}{0.0, 0.42, 0.24}

\title[Random walks with asymmetric interaction]{Ergodicity of a system of interacting random walks with asymmetric interaction}

\author{Luisa Andreis}
\address{Dipartimento di Matematica ``Tullio Levi Civita'',
	Universit\`a degli Studi di Padova,
	via Trieste 63,
	I-35121 Padova,
	Italy}
\email{andreis@math.unipd.it}

\author{Amine Asselah}
\address{LAMA, Universit\'e Paris-Est Cr\'eteil, 61 Av. G\'en\'eral de Gaulle, 94010 Cr\'eteil Cedex, France}
\email{amine.asselah@u-pec.fr}

\author{Paolo Dai Pra}

\address{Dipartimento di Matematica ``Tullio Levi Civita'',
	Universit\`a degli Studi di Padova,
	via Trieste 63,
	I-35121 Padova,
	Italy}
\email{daipra@math.unipd.it}



\date{\today}

\begin{document}

\begin{abstract}
We study $N$ interacting random walks on the positive integers. Each particle has drift $\delta$ towards infinity, a reflection at the origin, and a drift towards particles with lower positions. This inhomogeneous mean field system is shown to be ergodic only when the interaction is strong enough. We focus on this latter regime, and point out the effect of piles of particles, a phenomenon absent in models of interacting diffusion in continuous space.
\end{abstract}

\maketitle

\section{Introduction}\label{introduction}
We consider a particle system where the interaction, if strong enough, generates ergodicity.
More precisely, we consider a system of $N$ particles on the non-negative integers $\N$, which without interaction evolve as independent random walks, with a drift towards infinity. The interaction induces jumps towards zero, whose size depends on the specific model 
we consider, and whose rate is proportional to the fraction of particles with positions smaller than the jumping particle.

Let us start with the simplest model we consider.
There is a fixed number $N$ of particles on $\N$, where each particle $X^N_{i}$, for $i=1,\dots,N$, 
makes jumps of size $1$. If $X^N_{i}>0$, then it goes to
\be{model_small}
\begin{array}{ll}
X_{i}^N+1 &\text{ with rate } 1+\delta,\\
X_{i}^N-1 &\text{ with rate } 1+ \lambda \frac{1}{N}\sum_{k=1}^N\mathds{1}(X^N_{k}<X^N_{i}).
\end{array}
\ee
If $X^N_{i}=0$, then the only allowed jump is rightward. 
Here $\delta\geq 0$ indicates a bias rightward, while $\lambda \frac{1}{N}\sum_{k=1}^N\mathds{1}(X^N_{k}<X^N_{i})$ is a bias leftward. 
We call this model the {\it small jump model}, and we consider a large class of models where the leftward jump induced by the interaction term may have amplitude wider than $1$. One interpretation of these models is as follows. $N$ individuals, each associated with an integer valued {\em fitness},  have an intrinsic tendency to improve their fitness in time. However, each individual mimicking only the {\em worse than him} may worsen his fitness. The question is whether a strong interaction can prevent some individuals from improving forever, i.e. escape towards infinity.
At the outset, we make two remarks which we illustrate in the {\it small jump model}.
(i) The asymmetry in the drift produces an inhomogeneous system: the rightmost
particle, when alone on its site, has a net drift of about $\delta-\lambda$, whereas the leftmost particle has a positive drift $\delta$.
(ii) Particles piled up at the same site do not interact, and this produces a tendency for piles to spread rightward.

When $\lambda=0$, and any $N$, each particle system has no stationary measure. Indeed, it consists of random walks with a nonnegative drift $\delta\geq0$ and reflection at zero.  Our aim is to estimate the {\em critical} interaction strength, say $\lambda^*_N(\delta)$ for the $N$ particle system and $\lambda^*_{\infty}(\delta)$ for the nonlinear process, above which the system has a stationary measure. 
We focus on the simpler model described above since it dominates all others in stochastic ordering. In particular, 
ergodicity of the small jumps model implies ergodicity of all others.

\begin{theorem}\label{teo_exp_erg}
For $N\geq2$, $\delta\geq 0$, and $\lambda> 12\delta+8\delta^2$,
the process $X^N=(X^N_{1},\dots,X^N_{N})$ described in \eqref{model_small} is exponentially ergodic, i.e. there exists a unique probability measure $\pi^N$ on $\mathbb{N}^N$ and $\rho_N<1$  such that for any initial condition $x\in \mathbb{N}^N$ and time $t>0$ 
\begin{equation*}%\label{geom_ergod}
\|P^N_{x}((X^N_{1}(t),\dots,X^N_{N}(t))\in\cdot)-\pi^N\|_{TV}\leq C_N(x)\rho^t_N, 
\end{equation*}
where $C_N(x)$ is bounded, and $\|\cdot\|_{TV}$ is the total variation norm. 
\end{theorem}

Theorem~\ref{teo_exp_erg} gives an upper bound on $\lambda^*_N(\delta)$ uniform in $N$. It establishes for instance that, when $\delta=0$ any positive $\lambda$ ensures ergodicity. On the other hand, it is clear that for $\lambda\leq\delta$ the particle system is transient. By means of a Lyapunov function it is possible to establish a the lower bound on $\lambda^*_N(\delta)$.
\begin{theorem}\label{THM:lower_bound} For $N\geq 2$,  $\delta\geq 0$, and
\begin{equation*}%\label{lower_bound_eq}
\lambda < \big(1+\epsilon_N\big)2\delta,\quad\text{with}\quad  \epsilon_N:=\frac{N^2 (\delta+2)}{N(N-1)(\delta+2)-2\delta}-1\longrightarrow 0,
\end{equation*}
the process $X^N$ generated by \eqref{model_small} is transient.
\end{theorem}
Since the interaction is of mean-field type, we associate to \eqref{model_small} a nonlinear Markov process $\{X(t)\}_{t\geq0}$ whose possible transitions at time $t\geq0$ are as follows.
\be{rates_nonlinear}
\begin{array}{ll}
X(t)+1 &\text{ with rate } 1+\delta,\\
X(t)-1 &\text{ with rate } 1+ \lambda \mu_t[0,X(t)),
\end{array}
\ee
where $\mu_t$ is the law of $X(t)$ and, as in \eqref{model_small}, when $X(t)$=0, only the rightward jump is allowed. 

We give an upper bound and a lower bound on the critical value $\lambda^*_{\infty}(\delta)$.
\begin{theorem}\label{esistenza_nonlinear}
For $\delta\geq0$, and $\lambda>4\delta$, the nonlinear process\eqref{rates_nonlinear} has at least one stationary distribution. Moreover, for $\lambda \leq 2\delta$ there is no stationary distribution.
\end{theorem}

In Section~\ref{lower_bounds_N} we point out that similar models in the continuum and with diffusive dynamics, have been studied in \cite{JoMa08, JoRe13, Rey15}. In those papers the authors study systems of particles whose drift depends on the {\em cumulative distribution function} (CDF) of the empirical measure, that translates into a McKean-Vlasov process with a drift depending on the CDF of the law of the process itself. Along the line of the latter papers, we obtain explicit conditions 
for the ergodicity of continuous analogues of \eqref{model_small} and \eqref{rates_nonlinear}. 

Despite the same interacting mechanism, the continuous and the discrete model display a peculiar difference.
Indeed, in the discrete model the particles can form large clusters on a single site. When particles are on the same site, according to our description, they cannot interact and this interferes with ergodicity. On the other hand, the interaction prevents the particles from escaping to infinity and it favors the creations of clusters.
The lower bound in Theorem~\ref{THM:lower_bound} is strictly greater than the critical value of the continuum model, highlighting the different role played by the occurrence of piles in our case. We believe that this difference is substantial and gives rise to a non-trivial expression for $\lambda^*_{\infty}(\delta)$, unexpected by the analysis of the continuous model. 
In Section~\ref{Jackson_network_conjecture}, we exploit a link with {\em Jackson's Networks}. This gives sharper estimates on the critical values. In particular we derive the exact form of $\lambda^*_2(\delta)$. For $N>2$ the applicability of this method is still an open problem; however Jackson's Networks suggest heuristic computation leading 
to conjecture the critical interaction strength for all values of $N$ as follows. 
\begin{conj}\label{conj:N}
Fix $N\geq3$, the process $X^N$ in \eqref{model_small} is ergodic if, and only if, 
\begin{equation}
\label{condition_N_Y^N}
(1+\delta)^N<\prod_{k=1}^{N-1}(1+\lambda\frac{k}{N}).
\end{equation}
\end{conj}
Taking the limit as $N$ goes to infinity, a natural conjecture is the critical interaction strength for the nonlinear process.
\begin{conj}\label{conj:lim}
Fix $\delta\geq0$, then for all $\lambda$ such that
\[
(1+\frac{1}{\lambda})\ln\left(1+\lambda\right)-1>\ln\left(1+\delta\right), \]
the nonlinear process $X$ in \eqref{rates_nonlinear} has at least one stationary measure.
\end{conj}

The rest of the paper is organized as follows. 
The well-posedness of the process \eqref{rates_nonlinear}
and the link between \eqref{model_small} and \eqref{rates_nonlinear} is addressed
in Section~\ref{prop_chaos}, where we define a larger class of models,
and prove propagation of chaos.
In Section~\ref{SEC:upper_bound_N}, we give an upper bound on $\lambda^*_N(\delta)$, uniform in $N\geq2$. We determine sufficient conditions for the exponential ergodicity of the processes $X^N$, and for the tightness of the corresponding stationary distributions.
In Section~\ref{upper_bound_infinity} we focus on the existence of stationary measures for the nonlinear process in \eqref{rates_nonlinear}. 

\section{The model and propagation of chaos}\label{prop_chaos}
In this Section we properly define a class of interacting random walks which includes \eqref{model_small}, and describe the large-scale limit in terms of a {\em propagation of chaos} property. Since all results in this Section are rather standard, we omit the proofs.
We consider $N$ particles on the nonnegative integers, let $X^N=(X^N_1,\dots,X^N_N)$ $\in$ $\mathbb{N}^N$ be the vector of the particles' positions. Each particle has its own intrinsic dynamics, which is then perturbed by interaction. \\

The {\bf intrinsic dynamic} is given by a simple biased random walk independent of the other particles and reflected at zero. This is described by $2$ independent Poisson clocks for each particle, one with rate $1$, governing the leftward jump and the other with rate $1+\delta$, $\delta\geq0$, governing the rightward jump.\\

The {\bf interaction dynamic} is tuned by a parameter $\lambda>0$. Every pair of particles $(X^N_i,X^N_j)$ is activated with rate $\frac{\lambda}{N}\phi(X^N_i,X^N_j)$, where $\phi\colon\mathbb{N}^2\rightarrow [0,1]$ is a  bounded symmetric function.
 If the two particles have the same fitness level, i.e. $X^N_i=X^N_j$, then nothing happens. Otherwise, if for example $X^N_i<X^N_j$, then the most fit one (in the example $X^N_j$) is encouraged to worsen. This means that its fitness makes a leftward jump of size
$\psi(X^N_j,X^N_i)$, where $\psi\colon\mathbb{N}^2\rightarrow \mathbb{N}$ is a symmetric function such that $1\leq \psi(x,y)\leq x\vee y$ for all $(x, y)$ $\in\mathbb{N}^2$. \\

This class of dynamics includes \eqref{model_small}, by choosing $\phi(x,y) = \psi(x,y) \equiv 1$, as well other models of interest. For instance, choosing  $\phi(x,y) = 1$ and $\psi(x,y) = |y-x|$, the particle with highest position jumps to the position of the lowest. In the framework of a population of individuals where the position of each individual describes its fitness, this can be interpreted as the death of a particle that gives birth to a child whose level of fitness is equal the one of a \emph{less fit} individual chosen at random.

The above Markovian dynamics can be described  in terms of the infinitesimal generator $\mathcal{L}^N$, acting on bounded measurable function $f:\mathbb{N}^N\rightarrow\mathbb{R}$ in the following way:
{\small \begin{multline}\label{inf_gen_B}
\mathcal{L}^Nf(z)=\sum_{i=1}^N\left[\mathds{1}(z_i>0)(f(z-\delta_i) -f(z))+(1+\delta)(f(z+\delta_i)-f(z))\right]\\
+\frac{\lambda}{N}\sum_{i=1}^N \sum_{k=1}^N\mathds{1}(z_k<z_i)\phi(z_i, z_k)\left(f(z-\delta_i \psi(z_i,z_k))-f(z)\right),
\end{multline}}
where $\delta_i(k)=1$ if $k=i$, and 0 otherwise. Alternatively, dynamics can be seen as the solution of a system of \mbox{SDE}s: for $i=1,\dots, N$
{\small \begin{multline}\label{SDE_N}
\displaystyle{dX^N_{i}(t)=-\mathds{1}(X^N_{i}(t^-)>0)\int_0^{\infty}\mathds{1}_{[0,1]}(u)\mathcal{N}^i_{(-)}(du,dt)+\int_0^{\infty}\mathds{1}_{[0,1+\delta]}(u)\mathcal{N}^i_{(+)}(du,dt)}\\
\displaystyle{-\int_{[0,1]}\int_0^{\infty}\sum_{k=0}^{X^N_{i}(t^-)-1}\psi(k,X^N_{i}(t^-))\mathds{1}_{I_{k}(X^N_i(t^-),\mu^N_{t^-})}(h)\mathds{1}_{[0,\lambda]}(u)\mathcal{N}^i(du,dh,dt),}
\end{multline}}
where $\{\mathcal{N}^i_{(-)},\mathcal{N}^i_{(+)},\mathcal{N}^i\}_{i=1,\dots,N}$ are independent stationary Poisson processes with characteristic measures, respectively, $dudt$, $dudt$ and $dudhdt$, the empirical measure $\mu^N$ is defined by
\begin{equation*}
%\label{emp_measure}
\mu^N_t=\frac{1}{N}\sum_{i=1}^N\delta_{X^N_{i}(t)},
\end{equation*}
and the interval $I_k(x,\mu)$ is given as follows:
\[
I_{k}(x,\mu)\colon =  \left\{ \begin{array}{ll} \left(\sum_{y=0}^{k-1}\phi(y,x)\mu(y),\sum_{y=0}^{k}\phi(y,x)\mu(y)\right] & \mbox{for }  k>0 \\ \left(0,\phi(0,x)\mu(0)\right] & \mbox{for }  k=0 \end{array} \right.
\]
Note that the solution of \eqref{SDE_N} can be constructed {\em pathwise} for any initial condition, so no problem of well posedness arise here.

\subsection{Mean-field limit and propagation of chaos}

The heuristic limit $N \rightarrow \infty$ in \eqref{SDE_N}, leads to the following nonlinear \mbox{SDE}:
{\small \begin{multline}\label{SDE_limite}
dX(t)=-\mathds{1}(X(t^-)>0)\int_0^{\infty}\mathds{1}_{[0,1]}(u)\mathcal{N}_{(-)}(du,dt)+\int_0^{\infty}\mathds{1}_{[0,1+\delta]}(u)\mathcal{N}_{(+)}(du,dt)\\
-\int_{[0,1]}\int_0^{\infty}\sum_{k=0}^{X(t^-)-1}\psi(k,X(t^-))\mathds{1}_{I_k(X(t^-),\mu_{t^-})}(h)\mathds{1}_{[0,\lambda]}(u)\mathcal{N}(du,dh,dt),
\end{multline}}
where $\mu_t=\Law(X(t))$, $\{\mathcal{N}_{(-)},\mathcal{N}_{(+)},\mathcal{N}\}$ are independent stationary Poisson processes with characteristic measures, respectively, $dudt$, $dudt$ and $dudhdt$. The intervals $I_k(x,\mu)$ are defined as above. 
Henceforth, to ensure well-posedness of the nonlinear system, we require the functions $\psi$ and $\phi$ to satisfy the following condition: there exists $C<\infty$ such that for all $x,y$ $\in$ $\mathbb{N}$ and  $\alpha,\beta$ $\in$ $\mathcal{M}(\mathbb{N})$ probability measure on $\mathbb{N}$ 
$$
\left|\sum_{k=0}^{x\vee y-1}\psi(k,x\vee y)\left|I_k(x,\alpha)\Delta I_k(y,\alpha)\right|\right|\leq C|x-y|,$$
where for $A,B$ two intervals of the real line  $A\Delta B\colon=A \backslash B\cup B\backslash A$
and
$$
\left|\sum_{(x,y,z)\in\mathcal{A}}\alpha(y)\alpha(z)-\beta(y)\beta(z)\right|\leq C\sum_{x\in\mathbb{N}}|\alpha(x)-\beta(x)|,$$
where $\mathcal{A}\colon=\{(x,y,z)\in\mathbb{N}^3\colon z>x,\, z>y,\, z-\psi(y,z)=x\}$. These Lipschitz conditions are designed to allow the usual proof of uniqueness via Gronwall's inequality. Notice that they are obviously satisfies for the model in \eqref{model_small}.

\begin{theorem}[Propagation of chaos]\label{THM:prop_chaos} For every $\mu_0$ probability measure on $\mathbb{N}$, equation \eqref{SDE_limite} admits a pathwise unique solution whose law is denoted $\mu \in \mathcal{M}(\mathbf{D}(\mathbb{R}^+,\mathbb{N}))$. Moreover, let  $P^N$ $\in$ $\mathcal{M}(\mathbf{D}(\mathbb{R}^+,\mathbb{N})^N)$ be the law of the solution of system \eqref{SDE_N} with i.i.d. $\mu_0$-distributed initial conditions. Then the sequence $P^N$ is $\mu$-chaotic: for every $k \geq 1$, the projection of $P^N$ of the first $k$ components converges weakly, as $N \rightarrow \infty$, to the product measure $\mu^{\otimes k}$.
\end{theorem}

%{\red{Add references}}
The proof of propagation of chaos follows the classical approach developed in~\cite{Szn84} for proving that the sequence of empirical measures is tight, that its limit points belong to the solutions of \eqref{SDE_limite}, and that its solution is unique. This approach is more flexible than the coupling approach presented in~\cite{Szn91}, but it does not provide any rate of convergence. Since it is standard, we omit to reproduce the details.

\section{Exponential ergodicity of the particle systems: \\
proof of Theorem~\ref{teo_exp_erg}}\label{SEC:upper_bound_N}

In this section we study the long time behavior of the system with $N$ particles. We restrict the analysis to the specific model with {\em small jumps}, defined by \eqref{model_small}, whose generator is given by
\begin{equation*}%{SJgenerator}
\mathcal{L}^N_{(SJ)} f(z)=\sum_{i=1}^N\left(\mathds{1}(z_i>0)\nabla^{-}_if(z)+(1+\delta)\nabla^+_if(z)\right)+\frac{\lambda}{N}  \sum_{i=1}^N \nabla^{-}_if(z)\sum_{k=1}^N\mathds{1}(z_k<z_i),
\end{equation*}
where $\nabla^{-}_if(z)=f(z-\delta_i)-f(z)$ and $\nabla^{+}_if(z)=f(z+\delta_i)-f(z)$. It is easily seen that this model stochastically dominates all models defined in Section~\ref{prop_chaos}. In other words, let $Y^N(t)$ be any Markov process among those defined in Section~\ref{prop_chaos}. It can be coupled with $X^N(t)$ defined in \eqref{model_small} such that $Y^N(0) = X^N(0)$ and $Y^N(t) \leq X^N(t)$ for every $t \geq 0$, with respect to the componentwise partial order on $\N^N$. By standard results on countable Markov chains, the ergodicity of $Y^N(t)$ follows from that of $X^N(t)$.\\

Our purpose is to prove Theorem~\ref{teo_exp_erg} by means of a Lyapunov function. 
We choose a function that is the product of two exponential functions, encoding two characteristics of the particle system: the center of mass and the height of the highest ``pile'' of particles. This function depends on two positive parameters $\alpha$ and $\beta$ that we tune in order to produce the desired inequality. Let us define 
$$
\psi(x)=\frac{1}{N}\sum_{i=1}^{N}e^{\alpha x_i},\quad\text{and}\quad
\phi(x)=e^{+\frac{\beta}{N}\bar\eta(x)}, \quad\text{where}\quad
 \bar\eta(x)\colon=\max_{v\in\mathbb{N}}\sum_{i=1}^N\1(x_i=v).
 $$
Then, $V^N_{\alpha,\beta}(x)\colon=\psi(x)\phi(x)$ is our candidate Lyapunov function.
Let us now describe briefly the idea of the proof. We exploit the multiplicative form of $V^N_{\alpha,\beta}(x)$ %=\psi(x)\phi(x)$ 
and the fact that we can write
$$
\mathcal{L}^N_{(SJ)} \psi\phi= \psi\mathcal{L}^N_{(SJ)} \phi +\phi\mathcal{L}^N_{(SJ)} \psi+\Gamma^N_{(SJ)}(\phi,\psi),
$$
where $\Gamma^N_{(SJ)}$ is the operator carr\'e du champ. Now, $\Gamma^N_{(SJ)} V^N_{\alpha,\beta}$ can be bounded by a term proportional to $(e^{\beta}-1)(e^{\alpha}-1)V^N_{\alpha,\beta}(x)$. For $\alpha$ sufficiently small and $\beta=C\alpha$ for an
appropriatly chosen constant $C>0$, we find $\gamma>0$  and a constant $H\geq0$ for which $\mathcal{L}^N_{(SJ)} V^N_{\alpha,\beta}(x)\leq -\gamma V^N_{\alpha,\beta}(x)+H$. This establishes the exponential ergodicity criterion of Meyn and Tweedie \cite{MeTw93}.

\begin{proof}[Proof of Theorem~\ref{teo_exp_erg}]
Fix $\delta\geq 0$ and $N\geq 2$. It is sufficient to prove that the exponential ergodicity criterion from Meyn and Tweedie, \cite{MeTw93} holds for all values of $\lambda$ greater than  $\lambda^*= \delta^2+12\delta$. Note that this $\lambda^*$ does not depend on the size $N$ of the particle system. %We develop details and computations for the infinitesimal generator $\mathcal{L}^N_{(SJ)}$, since the model (BJ) is treatable in the same way. 
%Let $\alpha,\,\beta$ be two positive parameters, we define, for all $x\in\mathbb{N}^N$ the function \begin{equation}\label{definition_V}V^N_{\alpha,\beta}(x)\colon=\phi(x)\psi(x)=\frac{1}{N}\sum_{i=1}^{N}e^{\alpha x_i}e^{+\frac{\beta}{N}\bar\eta(x)},\end{equation}where $\bar\eta(x)\colon=\max_{v\in\mathbb{N}}\sum_{i=1}^N\1(x_i=v)$.
We aim now to bound the following function 
$$
\mathcal{L}^N_{(SJ)} V^N_{\alpha,\beta}(x)=\phi(x)\mathcal{L}^N_{(SJ)} \psi(x)+\psi(x)\mathcal{L}^N_{(SJ)}\phi(x)+\Gamma^N_{(SJ)}(\phi,\psi)(x).
$$ 
We treat separately $\mathcal{L}^N_{(SJ)}\psi(x)$, $\mathcal{L}^N_{(SJ)}\phi(x)$ and $\Gamma^N_{(SJ)}(\psi,\phi)(x)$ and  we divide the space $\mathbb{N}^N$ into two subsets, where we use two different approaches. One subset of $\mathbb{N}^N$ is the region of space such that where there is {\bf one single tall pile of particles  } (by {\bf tall pile} we intend that it contains more than half the particles), i.e. the region 
$$
\Lambda_N\colon=\{x\in \mathbb{N}^N\, \colon\,\bar \eta(x)>\frac{N}{2} \}.
$$
The other region is its complementary  $\Lambda_N^c$. Note that in $\Lambda_N$ there is only one tall pile.
The bound on $\mathcal{L}^N_{(SJ)} \psi(x)$ relies on two inequalities. First,
 $$
K_N:= \frac{1}{N}\sum_{i=1}^N\mu_N[0,x_i)=\frac{1}{2N^2}\sum_{i,j=1}^N\1(x_j\neq x_i)\geq \frac{\bar\eta(x)}{N}\big(1-\frac{\bar\eta(x)}{N}\big).
$$
Thus, on $\Lambda_N$, $K_N\ge \big(1- \frac{\bar\eta(x)}{N}\big)/2$.
Secondly, using FKG's inequality,
$$
\frac{1}{N}\sum_{i=1}^Ne^{\alpha x_i}\mu_N[0,x_i)\geq \psi(x)\frac{1}{N}\sum_{i=1}^N\mu_N[0,x_i)
$$
We start now estimating $\mathcal{L}^N_{(SJ)} \psi$,
\begin{multline*}
 \,\mathcal{L}^N_{(SJ)} \psi(x)=\sum_{i=1}^N (1+\delta)\nabla^+_i\psi(x)+\nabla^-_i\psi(x)-\sum_{i=1}^N\1(x_i=0)\nabla^-_i\psi(x)+\lambda\sum_{i=1}^N\mu_N[0,x_i)\nabla^-_i\psi(x)\\
=(e^{\alpha}+e^{-\alpha}-2)\psi(x)+\delta(e^{\alpha}-1)\psi(x)\\
+(1-e^{-\alpha})\frac{1}{N}\sum_{i=1}^N\1(x_i=0)-\lambda(1-e^{-\alpha})\frac{1}{N}\sum_{i=1}^Ne^{\alpha x_i}\mu_N[0,x_i)\\
\leq(e^{\alpha}+e^{-\alpha}-2)\psi(x)+\delta(e^{\alpha}-1)\psi(x)\\
+(1-e^{-\alpha})\frac{1}{N}\sum_{i=1}^N\1(x_i=0)-\lambda(1-e^{-\alpha})\psi(x)K_N\mathds{1}(\Lambda_N^c)-\lambda(1-e^{-\alpha})\psi(x)\frac{1-\frac{\bar\eta(x)}{N}}{2}\mathds{1}(\Lambda_N).
\end{multline*}
%We give a lower bound to the quantity $K\doteq \frac{1}{N}\sum_{i=1}^N\mu^N_{(SJ)}[0,x_i)$ in terms of $\bar \eta$. Notice that this term can be rewritten as the number of the unordered pairs of particles in distinct positions, 
 %$$\frac{1}{N}\sum_{i=1}^N\mu^N_{(SJ)}[0,x_i)=\frac{1}{2N^2}\sum_{i,j=1}^N\1(x_j\neq x_i)\geq \frac{1-\frac{\bar\eta}{N}}{2}.$$
%We will use this bound when $\bar \eta>\frac{N}{2}$.\\
The bound on $\mathcal{L}^N_{(SJ)} \phi(x)$, instead, is performed as follows.
\begin{itemize}
\item[ i)] For all $x$ $\in$ $\Lambda_N$ there exists a unique $v^*(x)=\arg\max_{v\in\mathbb{N}}\sum_{i=1}^N\1(x_i=v)$, so the function $\phi(x)$ changes values under the effect of $\mathcal{L}^N_{(SJ)}$ only because of the particles in $v^*(x)-1$, $v^*(x)$ and $v^*(x)+1$.  %notice that $\exists!$ $v^*$ $\in$ $\mathbb{N}$ such that $v^*=\arg\max_{v\in\mathbb{N}}\sum_{i=1}^N\1(x_i=v)$.
Therefore, in this case
\[
\begin{split}
\mathcal{L}^N_{(SJ)} \phi(x)  = &\sum_{i=1}^N (1+\delta)\nabla^+_i\phi(x)+\nabla^-_i\phi(x)-\sum_{i=1}^N\1(x_i=0)\nabla^-_i\phi(x)+\lambda\sum_{i=1}^N\mu_N[0,x_i)\nabla^-_i\phi(x)\\
  = &-\bar \eta(x)(1-e^{-\beta/N}) \left(1+\delta+\1(v^*(x)>0)+\lambda\mu_N[0,v^*(x))\right)\phi(x)\\
 & +(e^{\beta/N}-1)(\eta(v^*(x)-1)(1+\delta)+\eta(v^*(x)+1)(1+\lambda\mu_N[0,v^*(x)+1)))\phi(x)\\
\leq & \left[- \frac{\bar\eta(x)}{N}N(1-e^{-\beta/N})(1+\delta)- \frac{\bar\eta}{N}N(1-e^{-\beta/N})\lambda\mu_N[0,v^*(x))\right.\\
& +(N-\bar\eta(x))\lambda\mu_N[0,v^*(x))(e^{\beta/N}-1)+(1+\delta)(N-\bar\eta(x))(e^{\beta/N}-1) \\ & \left.+(N-\bar\eta(x))(e^{\beta/N}-1)\lambda\frac{\bar\eta(x)}{N}\right]\phi(x).
\end{split}
\]
\item[ii)] In $\Lambda_N^c$, we bound $\mathcal{L}^N_{(SJ)} \phi(x)$ with the pessimistic assumption that every jump increases $\phi(x)$ by an amount $(e^{\beta/N}-1)\phi(x)$, this means that
$$
\mathcal{L}^N_{(SJ)}\phi(x)\leq \left(N(2+\delta)+\lambda\sum_{i=1}^N\mu_N[0,x_i)\right)(e^{\beta/N}-1)\phi(x),
$$ 
where the term $NK_N$ appears and it will compensate the same term coming from $\mathcal{L}^N_{(SJ)} \psi(x)$.
\end{itemize}

The {\it carr\'e du champ} term reads
\[
\begin{split}
\Gamma^N_{(SJ)}(\phi,\psi)(x) = & \sum_{i=1}^N (1+\delta)\nabla^+_i\psi(x)\nabla^+_i\phi(x)+\nabla^-_i\psi(x)\nabla^-_i\phi(x) \\ & -\sum_{i=1}^N\1(x_i=0)\nabla^-_i\psi(x)\nabla^-_i\phi(x)+\lambda\sum_{i=1}^N\mu_N[0,x_i)\nabla^-_i\psi(x)\nabla^-_i\phi(x).
\end{split}
\]
It admits the following elementary bound:
$$
|\Gamma^N_{(SJ)}(\psi,\phi)(x)|\leq N(2+\lambda+\delta)(e^{\alpha}-1)(e^{\beta/N}-1)V^N_{\alpha,\beta}(x).
$$
Given these bounds, we want to identify those values for $\lambda$, for which we can properly choose $\alpha$, $\beta$ positive such that 
$$
\mathcal{L}^N_{(SJ)} V^N_{\alpha,\beta}(x)\leq -\gamma_N V^N_{\alpha,\beta}(x)+H,
$$
 for some constants $\gamma_N>0$ and $H\geq0$. In the two complementary regions we have the following upper bounds for $\mathcal{L}^N_{(SJ)} V^N_{\alpha,\beta}(x)$, up to bounded terms that can be incorporated in $H$:
\begin{itemize}
\item[{\bf A)}] for $x$ $\in$ $\Lambda_N$:
\begin{multline*}
\left[(e^{\alpha}+e^{-\alpha}-2)+\delta(e^{\alpha}-1)-\lambda(1-e^{-\alpha})\frac{1-\frac{\bar\eta(x)}{N}}{2}\right.\\
-\Big( \frac{\bar\eta(x)}{N}e^{-\beta/N}-(1-\frac{\bar\eta(x)}{N})\Big) \lambda (e^{\beta/N}-1)\mu_N[0,v^*(x))\\
- \frac{\bar\eta(x)}{N}N(1-e^{-\beta/N})(1+\delta)+(1+\delta)(N-\bar\eta(x))(e^{\beta/N}-1)+(N-\bar\eta(x))(e^{\beta/N}-1)\lambda\frac{\bar\eta(x)}{N}\\
\left.+N(2+\delta+\lambda)(e^{\beta/N}-1)(e^{\alpha}-1)\right] V^N_{\alpha,\beta}(x);
\end{multline*}
\item[{\bf B)}]for $x$ $\in$ $\Lambda_N^c$:
\begin{multline*}
\left[(e^{\alpha}+e^{-\alpha}-2)+\delta(e^{\alpha}-1)
-\lambda(1-e^{-\alpha})K_N+\left(N(2+\delta)+\lambda N K_N\right)(e^{\beta/N}-1)\right.\\
\left.+N(2+\delta+\lambda)(e^{\beta/N}-1)(e^{\alpha}-1)\right] V^N_{\alpha,\beta}(x).
\end{multline*}
\end{itemize}
We want to make negative the two terms above within the square brackets, by choosing properly $\alpha$ and $\beta$; we start by letting $\beta=C\alpha$, for some $C>0$ and take $\alpha$ small. \\

\noindent Let us look at the quantity in {\bf A)}. The term 
$$
 \frac{\bar\eta(x)}{N}e^{-\beta/N}-(1-\frac{\bar\eta(x)}{N})
$$
is positive for $\beta$ sufficiently small, so we can neglect it. The terms  $N(2+\delta+\lambda)(e^{\beta/N}-1)(e^{\alpha}-1)$ and $(e^{\alpha}+e^{-\alpha}-2)$ are of order $\alpha^2$ for $\alpha \downarrow 0$; they can be neglected since the remaining terms are of order $\alpha$. We are left to find $\lambda$ and $C$ such that 
{\small
$$
\delta(e^{\alpha}-1)
-\lambda(1-e^{-\alpha})\frac{1-\xi}{2} - \xi N(1-e^{-\beta/N})(1+\delta)+(1+\delta)(1-\xi)N(e^{\beta/N}-1)+(1-\xi)N(e^{\beta/N}-1)\lambda\xi$$
}
is negative for all $\xi$ $\in$ $(\frac{1}{2},1]$. Then, for $\alpha$ sufficiently small,  this condition becomes 
$$\delta-(1-\xi)\left(\frac{\lambda}{2}-C(1+\lambda\xi+\delta)\right)-C\xi(1+\delta)<0,$$
for all $\xi\in(1/2,1]$, that gives the conditions on $C$:
\begin{equation*}
\frac{\delta}{1+\delta}\leq C \leq \frac{\lambda-4\delta}{\lambda}.
\end{equation*}
\noindent Now we look at point {\bf B)}. Again, we do not consider the terms $N(2+\delta+\lambda)(e^{\beta/N}-1)(e^{\alpha}-1)$ and $(e^{\alpha}+e^{-\alpha}-2)$. We want to find conditions under which 
$$
\delta(e^{\alpha}-1)
-\lambda(1-e^{-\alpha})K_N+\left(N(2+\delta)+\lambda N K_N\right)(e^{\beta/N}-1)
$$
is negative for all values assumed by $K_N$ for $x$~$\in$~$\Lambda_N^c$. This, for $\alpha$ small, is 
$$
\delta-\lambda K_N+C(2+\delta+\lambda K_N)\leq0,
$$
that gives an additional conditions on $C$: for every $k$ $\in$ $[1/4,1]$
\begin{equation*}
C\leq \frac{\lambda k-\delta}{2+\delta+\lambda k}.
\end{equation*}
The conditions are independent of $N$, and they are satisfied only if $\lambda > 12\delta+8\delta^2$.

\end{proof}

\section{Invariant measures for the nonlinear process:\\
proof of Theorem~\ref{esistenza_nonlinear}}\label{upper_bound_infinity}

The ergodicity of each $N$-particle system \eqref{inf_gen_B}, together with the propagation of chaos stated in Theorem \ref{THM:prop_chaos}, do not directly imply ergodic properties on the limiting dynamics \eqref{SDE_limite}, not even the existence of an invariant measure. A stronger ({\em time-uniform}) propagation of chaos property would be needed for this purpose (see e.g. \cite[Theorem~3.1]{CaGuMa08}), but this result is not proved yet. Thus we study separately the nonlinear system. Again, we focus on the model with small jumps, described in \eqref{rates_nonlinear}, that corresponds to the solution $\{X(t)\}_{t\geq0}$ of the following nonlinear \mbox{SDE}
{\small \begin{multline}\label{SDE_limiteS}
d X(t)=-\mathds{1}( X(t^-)>0)\int_0^{\infty}\mathds{1}_{[0,1]}(u)\mathcal{N}_{(-)}(du,dt)+\int_0^{\infty}\mathds{1}_{[0,1+\delta]}(u)\mathcal{N}_{(+)}(du,dt)\\-\int_{[0,1]}\int_0^{\infty}\1_{[0,\mu_{t^-}([0, X(t^-)))}(h)\mathds{1}_{[0,\lambda]}(u)\mathcal{N}(du,dh,dt),\end{multline}} 
where $\mu_t=\Law(X(t))$, $\{\mathcal{N}_{(-)},\mathcal{N}_{(+)},\mathcal{N}\}$ are independent stationary Poisson processes with characteristic measures, respectively, $dudt$, $dudt$ and $dudhdt$, and, by convention $\mu_t[0,0)=0$. 

We first prove, under the condition $\l > 4\d$, the existence of at least one stationary distribution by means of a transformation $\Gamma$ in the space $\mathcal{M}(\mathbb{N})$, for which every stationary distribution of \eqref{SDE_limiteS} is a fixed point. This is an approach widely exploited in the study of quasi-stationary distributions (QSD) in countable spaces, see \cite{AsCa03, FeKeMaPi95,  FeMa07}. \\

We define a continuous time Markov chain on $\mathbb{N}$, parametrized by a measure. Fix $\mu$ $\in$ $\mathcal{M}(\mathbb{N})$, then let $\{X^{\mu}(t)\}_{t\geq 0}$ be the process with infinitesimal generator defined as follows. For $f$ $\in$ $C_b$, and $x\in \mathbb{N}$
\begin{equation*}
\mathcal{L}^{\mu}f(x)=(1+\delta)(f(x+1)-f(x))+\mathds{1}(x>0)(1+\lambda \mu[0,x))(f(x-1)-f(x)).
\end{equation*}
Assuming $\lambda > \delta$, for every measure $\mu$, the \emph{birth and death} process $\{X^{\mu}(t)\}_{t\geq0}$ is ergodic, 
and $\pi^{\mu}$ denotes its unique stationary distribution. Define the  map
\begin{equation*}
\begin{array}{rccc}
\Gamma\colon &\mathcal{M}(\mathbb{N}) &\rightarrow &\mathcal{M}(\mathbb{N})\\
&\mu &\mapsto &\pi^{\mu},
\end{array}
\end{equation*}
%We know that $\pi^{\mu}$ satisfies, for all function $f$, 
%\[
%\sum_{x=0}^{\infty}\mathcal{L}^{\mu}f(x)\pi^{\mu}(x)=0.\]
By definition, $\mu^*$ is a stationary distribution for \eqref{SDE_limiteS} if and only if it is a fixed point of $\Gamma$.
\begin{proof}[Proof of Theorem~\ref{esistenza_nonlinear}: upper bound] The proof of the upper bound consists of three steps. First we define an auxiliary map that stochastically dominates the map $\Gamma$, then we prove that this map preserves a certain subset of $\mathcal{M}(\mathbb{N})$, finally we prove that $\Gamma$ admits at least one fixed point in that subset.\\

\emph{Step 1.} 
Given $ \mu \in \mathcal{M}(\mathbb{N})$, consider the birth and death process with infinitesimal generator
\begin{equation*}
\mathcal{L}_mf(x)=(1+\delta)(f(x+1)-f(x))+(\mathds{1}(x>0)+\frac{\lambda}{2}\mathds{1}(x> m))(f(x-1)-f(x)),
\end{equation*}
Since we are assuming $\lambda > 4 \delta$ (here  $\lambda > 2 \delta$ would suffice), this process is ergodic, and we denote by $\pi_{m}$ its stationary distribution. We claim that for all $\mu$ $\in$ $\mathcal{M}(\mathbb{N})$, we have $\pi^{\mu} \preceq \pi_{med(\mu)}$, where $med(\mu)$ denotes the median of $\mu$ and $\preceq$ is the usual stochastic ordering on $\mathcal{M}(\mathbb{N})$.
This is proved by using the so-called {\em basic coupling} between $\mathcal{L}^{\mu}$ and $\mathcal{L}_{med(\mu)}$, i.e. the Markov process $(X_t,Y_t)$ on $\mathbb{N}^2$ that, at every time $t\geq0$, jumps in the following positions:
\begin{equation*}
\begin{array}{lcl}
(X_t+1,Y_t+1) &\text{ with rate }& 1+\delta,\\
(X_t-1,Y_t-1) &\text{ '' } &\mathds{1}(X_t>0)\mathds{1}(Y_t>0)+\lambda\left(\mu[0,X_t)\wedge \frac{\mathds{1}(Y_t>med(\mu))}{2}\right),\\
(X_t-1,Y_t) &\text{ '' }& \mathds{1}(X_t>0)\mathds{1}(Y_t=0)+\lambda\left(\mu[0,X_t)-\frac{\mathds{1}(Y_t>med(\mu))}{2}\right)_+,\\
(X_t,Y_t-1) &\text{ '' } &\mathds{1}(X_t=0)\mathds{1}(Y_t>0)+\lambda \left(\frac{\mathds{1}(Y_t>med(\mu))}{2}-\mu[0,X_t)\right)_+.
\end{array}
\end{equation*}

Note that this dynamics preserves the order $X_t \leq Y_t$. Since $X$ evolves according to $\mathcal{L}^{\mu}$ and $Y$ to $\mathcal{L}_{med(\mu)}$, which are both ergodic, the order is preserved in equilibrium, i.e. $\pi^{\mu} \preceq \pi_{med(\mu)}$ as desired. We also observe that, by a similar (simpler) coupling argument, $\pi_m \preceq \pi_{m'}$ for $m \leq m'$.

\medskip

\emph{Step 2.}  
We now show that if $m^*$ is large enough and $\mu \preceq \pi_{m^*}$, then $\pi^{\mu} \preceq \pi_{m^*}$. By Step 1, this follows if we show that $\pi_{med(\mu)} \preceq \pi_{m^*}$, which amounts to $med(\mu) \leq m^*$; since $\mu \preceq \pi_{m^*}$.Thus, it is enough to show that for some $m^*$, $med(\pi_{m^*}) \leq m^*$. To see this, we use  the explicit formula for the stationary measure of a birth and death process, obtained by the detailed balance equation: for $Z^*$ normalizing constant,
\begin{equation*}%\label{pi_med}
\left\{\begin{array}{ll}
\pi_{m^*}(x)=\frac{1}{Z^*}(1+\delta)^x \, & \text{ for }\, x\leq m^*;\\
\pi_{m^*}(x)=\frac{1}{Z^*}(1+\delta)^{m^*}\left(\frac{1+\delta}{1+\lambda/2}\right)^{x-m^*}\, & \text{ for }\, x>m^*.
\end{array}\right.
\end{equation*}
The desired inequality $med(\pi_{m^*}) \leq m^*$ follows if we show that
$$
\pi_{m^*}[0,m^*]>\pi_{m^*}(m^*,\infty).
$$
Indeed, this is equivalent to
$$
\frac{(1+\delta)^{\lfloor m^*\rfloor+1}-1}{\delta}>(1+\delta)^{\lfloor m^*\rfloor}\frac{1+\delta}{\lambda/2-\delta}$$
and, by simplifying,
$$
\frac{\lambda/2-2\delta}{\lambda/2-\delta}>\frac{1}{(1+\delta)^{\lfloor m^*\rfloor+1}},$$
which holds for $m^*$ sufficiently large. \\

\emph{Step 3.} 
Define the set
$$
\mathcal{M}_{m^*}(\mathbb{N})\colon=\left\{\mu\in\mathcal{M}(\mathbb{N}) \colon \mu\preceq \pi_{m^*}\right\},
$$
where $m^*$ has been determined in step 2. We have seen that the function $\Gamma$ maps $\mathcal{M}_{m^*}$ into itself. Moreover, $\mathcal{M}_{m^*}$ is clearly convex, and it is compact for the weak topology, being closed and tight. The existence of a fixed point follows from Schauder-Tychonov fixed point theorem if we show that $\Gamma$ is continuous. Let $\mu_n\rightarrow\mu$ in $\mathcal{M}_{m^*}$. By the formula for the stationary distribution of a birth and death process we have
\[
\pi^{\mu_n}(x)=\frac{1}{Z^*_n}\frac{(1+\delta)^k}{\prod_{h=0}^{k-1}(1+\lambda\mu_n[0,h))},
\]
with 
\[
Z^*_n\colon=\sum_{k=0}^{\infty}\frac{(1+\delta)^k}{\prod_{h=0}^{k-1}(1+\lambda\mu_n[0,h))}.
\]
Since
\[
\frac{(1+\delta)^k}{\prod_{h=0}^{k-1}(1+\lambda\mu_n[0,h))} \leq \frac{(1+\delta)^k}{\prod_{h=0}^{k-1}(1+\lambda\pi_{m^*}[0,h))},
\]
by the Dominated Convergence Theorem 
\[
Z^*_n \ra Z^* := \sum_{k=0}^{\infty}\frac{(1+\delta)^k}{\prod_{h=0}^{k-1}(1+\lambda\mu[0,h))},
\]
and $\pi^{\mu_n} \ra \pi^{\mu}$, which establishes continuity.
\end{proof}
Let us underline the importance of this approach with the fixed point argument. %Firstly, with a coupling argument similar to the one in the proof, \textcolor{red}{ this approach immediately gives the same result of existence of a stationary solution for equation \eqref{SDE_limite}.}  
It gives an upper bound for the critical value $\lambda^*_{\infty}(\delta)$ which is linear in $\delta$. Indeed, based on numerical computation, we see that the condition on $\lambda$ is not quadratic in $\delta$, as the one emerging from Theorem~\ref{teo_exp_erg}. Clearly the one found in Theorem~\ref{esistenza_nonlinear} is not optimal and in the following sections we propose conjectures for
 the critical value for the system (SJ) in both the particle system and the nonlinear limit equation.

\begin{proof}[Proof of Theorem~\ref{esistenza_nonlinear}: lower bound]
We show that, for $\lambda \leq 2 \delta$, the nonlinear system has no stationary distribution. Let us remark, to begin with, that for $\l \leq \d$ the conclusion is essentially obvious: indeed, the nonlinear Markov process can be coupled, monotonically from below, with a reflected random walk with forward rate $1+\d$ and backward rate $1+\l$, whose distribution at time $t$ tends to concentrate in $+\infty$ as $t \uparrow +\infty$, for any initial distribution. So assume $\l > \d$, and suppose there exists a stationary distribution $\mu$. The Markov process generated by $\mathcal{L}^{\mu}$ has a strictly negative drift for sufficiently large positions; this implies that its stationary distribution, that is $\mu$ by assumption, has tails not larger than exponentials. In particular, denoting by $(X_t)_{t \geq 0}$ the associated stationary process, $E(X_t)< +\infty$. Moreover, denoting by $\mbox{id}$ the identity map on $\mathbb{N}$,
\begin{equation} \label{lwbn}
0 = \frac{d}{dt}E(X_t) = E\left[\mathcal{L}^{\mu}\mbox{id}(X_t) \right] = \d - \l  \sum_{x \geq 1} \mu[0,x-1] \mu(x).
\end{equation}
But
\[
\begin{split}
\sum_{x \geq 1} \mu[0,x-1] \mu(x) & = \sum_{x \geq 1} \mu[0,x-1] \left( \mu[0,x]-\mu[0,x-1]\right) \\ & = \sum_{x \geq 1} \left(\mu^2[0,x] - \mu^2[0,x-1] \right) - \sum_{x \geq 1}\mu[0,x] \left( \mu[0,x]-\mu[0,x-1]\right) \\ 
& = 1 - \sum_{x \geq 1} \mu[0,x-1] \left( \mu[0,x]-\mu[0,x-1]\right) - \sum_{x \geq 0} \mu^2(x)
\end{split}
\]
which implies 
\[
\sum_{x \geq 1} \mu[0,x-1] \mu(x) < \frac12.
\]
Inserting this in \eqref{lwbn}, we get $\l > 2\d$, which completes the proof.
\end{proof}

\section{Lower bounds on the critical value $\lambda^*_N(\delta)$ for (SJ):\\
proof of Theorem~\ref{THM:lower_bound}}\label{lower_bounds_N}

\subsection{The continuous analogue}
Before giving the proof of Theorem~\ref{THM:lower_bound}, we briefly illustrate what is known for a similar model in the continuum (see \cite{JoMa08} for further detail). We consider, more specifically, the Markov process $(X_t)_{t \geq 0}$  in $D_N\colon=\{x\in\mathbb{R}^N: \, x_i\geq0\, \forall \,i=1,\dots,N\}$ with infinitesimal generator
\begin{equation}\label{inf_gen_c}
\mathbf{L}_c^Nf(x)=\sum_{i=1}^N\frac{1}{2}\frac{\partial^2}{\partial x_i^2}f(x)-\left(\delta-\frac{\lambda}{N}\sum_{k=1}^N\mathds{1}(x_k\leq x_i)\right)\frac{\partial}{\partial x_i}f(x),
\end{equation}
and reflection on the boundary of $D_N$. Optimal ergodicity conditions for this system follow from known results on reflecting Brownian motions in polyhedra, see \cite{Wil87, PaPi08}. The stationary distribution is explicit, and it is better described in terms of the {\em reordered process} $(X^N_{(1)}(t),\dots,X^N_{(N)}(t))$ obtained by ordering increasingly $(X^N_1(t),\dots,X^N_N(t))$, and that takes values in the {\em wedge} $\mathcal{W}_N := \{y \in \mathbb{R}^N: \, 0 \leq y_1 \leq y_2 \leq \cdots \leq y_N\}$.
\begin{proposition}\label{continuo}
The process given in \eqref{inf_gen_c} has a unique stationary distribution $\pi^N$ if and only if 
\begin{equation*}
\lambda>2\delta\frac{N}{N-1}=\colon \lambda^c_N(\delta).\end{equation*}
Moreover, this stationary distribution is such that the gaps $(X^N_{(1)},X^N_{(2)}-X^N_{(1)},X^N_{(3)}-X^N_{(2)},\dots)$ are independent exponential random variables of parameters $2a_i$, where
$$a_i=\frac{\lambda}{2 N} \left[(N+1-i)\left(i-\frac{\lambda(2-N)+2\delta N}{\lambda}\right)\right].$$
\end{proposition}

Let us stress that, in the case of diffusion processes reflected in a polyhedra, there is a.s. no triple collision, see \cite{PaPi08}. This means that the non-smooth parts of the boundary of the wedge $\mathcal{W_N}$ are of no importance in the dynamics of the reordered process, and that it is sufficient to consider reflection conditions on the hyperplanes of dimension $N-1$. This is the main difference between the continuous case  and our model. In the discrete case, indeed, the ``piles'' of particles (that correspond to the whole boundary of $D_N$) matter. In fact, the lower bound give in Theorem~\ref{THM:lower_bound} is {\em strictly larger} than $\lambda^c_N(\delta)$.
 \\

\subsection{Proof of Theorem~\ref{THM:lower_bound}}

Let $\mathcal{W}_N := \{y \in \mathbb{N}^N: \, 0 \leq y_1 \leq y_2 \leq \cdots \leq y_N\}$ be the state space of the reordered process $(X^N_{(1)}(t),\dots,X^N_{(N)}(t))$, and denote by $\mathcal{L}^N_{ord}$ its infinitesimal generator. 
The proof of this lower bound is made by means of a Lyapunov function. We define a linear function $f\colon \mathcal{W}_N\rightarrow\mathbb{R}$ such that for all $\lambda$ strictly greater than the lower bound
 \begin{equation*}%\label{Lyap_trans}
 \mathcal{L}^N_{ord} f(x)>0,
 \end{equation*}
 for all $x$ $\in$ $\mathcal{W}_N$. This implies transience of the Markov chain, see \cite[Theorem~2.2.7]{FaMaMe95}.

% 
% 
% The idea for the construction of this Lyapunov function is the following. Firstly we consider the barycenter of the particle system $B(x)=\sum_{k=1}^Nx_k$ as Lyapunov function. We see that, if $\lambda<\lambda^c_N(\delta)$, $\mathcal{L}^N_{ord} B(x)>0$ for all $x$~$\in$~$D_N$. If the parameters satisy this condition, the mean drift of the lowest $\lfloor\frac{N}{2}\rfloor$ particles is positive, i.e. for $k=1,\dots,\lfloor\frac{N}{2}\rfloor$ it holds $\mathcal{L}^N_{ord} x_k>0$ for all $x$ in the interior of $\mathcal{W}_N$. Then, we may consider, instead of the classic barycenter, a modification of it that gives more ``weight'' to the first particles. To get this lower bound, we consider the vector $v_{\epsilon}=(1+\epsilon,1,1,1,\dots)$ and define the Laypunov function $f_{\epsilon}(x)=\langle v_{\epsilon},x\rangle=(1+\epsilon)x_1+\sum_{k=2}^Nx_k$ for some $\epsilon>0$. We look for the maximal $\epsilon>0$ that improves the condition on $\lambda$, i.e. there exists $\rho(\epsilon)>0$ such that $\mathcal{L}^N_{ord} f_{\epsilon}(x)>0$ for all $x$~$\in$~$\mathcal{W}_N$, for every $\lambda<\lambda^c_N(\delta)+\rho(\epsilon)$.\\
%

We fix $N\geq2$ and $\delta\geq0$, then we consider the $N$ dimensional vector $v_{\epsilon}=(1+\epsilon,1,1,\dots,1)$ and the function $f_{\epsilon}(x)=\langle v_{\epsilon},x\rangle := \sum_i v_i x_i$, defined on $\mathcal{W}_N$. For $x$ in the interior of $\mathcal{W}_N$, i.e. when particles are in distinct positions, a simple computation leads to 
\begin{equation*}
 \mathcal{L}^N_{ord} f_{\epsilon}(x) =  N\delta +\epsilon \delta -\lambda \frac{N-1}{2}.
\end{equation*}
We decompose the boundary $\partial \mathcal{W}_N$ as follows:
\[
\partial \mathcal{W}_N = \bigcup_{k=1}^N \mathcal{W}_{(N,k)},
\]
where
\[
\mathcal{W}_{(N,k)}\colon=\{x\in \partial\mathcal{W}_N\colon x_1=\dots=x_k<x_{k+1}\},
\]
meaning that the lowest particle belongs to a pile of height $k$. For $x \in \mathcal{W}_{(N,1)}$ we have
\[
 \mathcal{L}^N_{ord} f_{\epsilon}(x) \geq  N\delta +\epsilon \delta -\lambda \frac{N-1}{2}.
 \]
Indeed, the only difference with respect to the interior of $\mathcal{W}_N$ is that some  particle of position $x_i$, with $i \geq 2$, has a rate of backward jump lower than $1+\lambda \frac{i-1}{N}$, due to the fact that other particles have the same position.  For $x \in \mathcal{W}_{(N,k)}$, with $k \geq 2$, the situation is different since the position $x_1$ may only decrease with a single jump and the rate of this move is proportional to the height of the first pile. This leads to the estimate
\begin{equation}\label{cond_lambda-k}
\mathcal{L}^N_{ord} f_{\epsilon}(x)\geq N\delta-k\epsilon-\lambda\frac{N-1}{2}+\lambda\frac{k(k-1)}{2N}. 
\end{equation}
It is easy to check that if $0 \leq \epsilon \leq \frac{3\lambda}{4N}$ then the minimum over $k$ in \eqref{cond_lambda-k} is attained at $k=2$. So, under the condition $0 \leq \epsilon \leq \frac{3\lambda}{4N}$, the inequality $ \mathcal{L}^N_{ord} f_{\epsilon}(x) > 0 $ follows for every $x \in \mathcal{W}_N$ from
\begin{equation} \label{eq-epsilon}
\begin{split}
N\delta +\epsilon \delta -\lambda \frac{N-1}{2} & >0 \\ 
N\delta +\epsilon \delta -\lambda \frac{N-1}{2} + \frac{\lambda}{N} - 2 \epsilon & >0.
\end{split}
\end{equation}
In the case $N\delta -\lambda \frac{N-1}{2} >0$ (which is exactly the condition of non-ergodicity in Proposition \ref{continuo}) we can chose $\epsilon = 0$. Otherwise, after having noticed that the second inequality in \eqref{eq-epsilon} implies $\epsilon \leq \frac{\lambda}{2N}$, the existence of a nonnegative $\epsilon$ for which \eqref{eq-epsilon} holds is equivalent to
\[
\frac{\lambda(N-1) - 2N\delta}{\delta} < N\delta -\lambda \frac{N-1}{2} + \frac{\lambda}{N},
\]
which yields
\[
\lambda < \frac{2 N^2 (\delta+2)\delta}{N(N-1)(\delta+2)-2\delta},
\]
which is the desired estimate.

%our aim is to find an admissible $\epsilon>0$, that increases the maximal $\lambda$ satisfying \eqref{Lyap_trans}. We need to check the admissibility of $\epsilon$ in those regions of $\mathcal{W}_N$ where the mean drift of the first particle is negative, that are the regions where the first particle belongs to a pile. Let first consider the subset $\mathcal{W}_{(N,k)}\colon=\{x\in\mathcal{W}_N\colon x_1=\dots=x_k<x_{k+1}\}$, where the first particle belongs to a pile of exactly height $k$, for $k=2,\dots, N$. Then we say that, for all $k=2,\dots,N$ and for all $x$~$\in$~$\mathcal{W}_{(N,k)}$, it holds
%$$
%\mathcal{L}^N_{ord} f_{\epsilon}(x)\geq N\delta-k\epsilon-\lambda\frac{N-1}{2}+\lambda\frac{k(k-1)}{2N}. 
%$$
%Therefore we want to find the maximal $\epsilon>0$ such that 
%$$\min_{k=2,\dots,N}N\delta-k\epsilon-\lambda\frac{N-1}{2}+\lambda\frac{k(k-1)}{2N}\geq0,$$
%that is $\epsilon_{max}=\frac{\lambda}{N(\delta+2)}$. Substituting $\epsilon_{max}$, we get that \eqref{cond_lambda} holds for all $\lambda<\frac{2 N^2 (\delta+2)\delta}{N(N-1)(\delta+2)-2\delta}$. 
%

\section{Conjectures on the exact critical values and stationary measure for (SJ)}\label{Jackson_network_conjecture}

\subsection{The gap process}
With a simple linear transformation of the process $(X^N_{(1)},\dots,X^N_{(N)})$, we define the {\bf gap process} $G^N=(G^N_{1},\dots,G^N_{N})$, where $G^N_1\colon=X^N_{(1)}$ and $G^N_i\colon=X^N_{(i)}-X^N_{(i-1)}$ for $i=2,\dots,N$, that is  a reflected random walk in $\mathbb{N}^N$. 
%and it will be the process that we study to find the stationary measure and the conditions for process $X^N$. 
In the continuous analogue, this process is a diffusion reflected in $R^N_+$ and we see from Proposition~\ref{continuo} that is possible to identify its stationary measure for each fixed $N$. In the stationary regime the gaps are independent and exponentially distributed with different parameters, such that the admissibility of these parameters determines the critical value of $\lambda$. We do not expect independence of the gaps for all $N\geq2$ in the discrete setting, because of the importance of triple (or more) collisions of particles. In the following we give a complete treatment in the case $N=2$ and we conjecture the critical value $\lambda^*_N(\delta)$ for $N>2$, based on the theory of Jackson networks.

\subsection{Exact study of gap process for $N=2$}
Let us focus on the case $N=2$. The gap process $G^2=\{(G^2_1(t),G^2_2(t))\}_{t\geq0}$ is a reflected random walk in the positive quadrant. It jumps from $\mathbf{g}=(g_1,g_2)$ according to the following rules.
\begin{equation*}
\begin{array}{lclcl}
\text{If } g_1>0, \, g_2>0,\quad \mathbf{g}  &\rightarrow & \mathbf{g}+(1,-1) & \text{ with rate }& 1+\delta\\
&  &  \mathbf{g}+(0,-1) & \text{ '' }& 1+\frac{\lambda}{2}\\
&&  \mathbf{g}+(-1,1) & \text{ '' }& 1\\
& &  \mathbf{g}+(0,1) & \text{ '' }& 1+\delta\\
&&&&\\
\text{If } g_1=0, \, g_2>0,\quad \mathbf{g}  &\rightarrow & \mathbf{g}+(1,-1) & \text{ '' }& 1+\delta\\
&  &  \mathbf{g}+(0,-1) & \text{ '' }& 1+\frac{\lambda}{2}\\
& &  \mathbf{g}+(0,1) & \text{ '' }& 1+\delta\\
&&&&\\
\text{If } g_1>0, \, g_2=0,\quad \mathbf{g} & \rightarrow  &  \mathbf{g}+(-1,1) & \text{ '' }& 2\\
& &  \mathbf{g}+(0,1) & \text{ '' }& 2+2\delta\\
&&&&\\
(0,0)& \rightarrow&  (0,1) & \text{ '' }& 2+2\delta.\\
\end{array}
\end{equation*} 

\begin{theorem}\label{Exact_N=2}
The process $G^2$ %has a unique stationary distribution $\pi^2$ 
is exponentially ergodic if and only if $\lambda>2\delta^2+4\delta$. Moreover, when it exists, the unique stationary measure $\pi_2$ has the following explicit form:
\begin{equation*}
\begin{array}{lr}
\pi_2(0,0)=\frac{C}{2}&\\
\pi_2(0,y)=C \left(\frac{1+\delta}{1+\frac{\lambda}{2}}\right)^y& y\geq1,\\
\pi_2(x,0)=\frac{C}{2} \left(\frac{(1+\delta)^2}{1+\frac{\lambda}{2}}\right)^x& x\geq1,\\
\pi_2(x,y)=C \left(\frac{(1+\delta)^2}{1+\frac{\lambda}{2}}\right)^x \left(\frac{1+\delta}{1+\frac{\lambda}{2}}\right)^y& x\geq1, y\geq1,\\
\end{array}
\end{equation*}
for $C\colon=\frac{2(\frac{\lambda}{2}-\delta)(\frac{\lambda}{2}-2\delta-\delta^2)}{(\frac{\lambda}{2}+\delta2)(\frac{\lambda}{2}+1)}$.
\end{theorem}
The proof of exponential ergodicity is based on the link between the gap process $G^2$ and a Jackson network, for which exponential ergodicity is proven by Fayolle, Malyshev and Menshikov, \cite{FaMaMe95}. Jackson networks are queueing models, first introduced by Jackson \cite{Jac63}, that proved the product form of their stationary distribution. An open Jackson network with two nodes $Z^2(t)\colon=(Z^2_1(t),Z^2_2(t))$ represents at time $t\geq0$ the length of two queues, where the inputs are Poissonian with parameters $\lambda_i$ at node $i$, for $i=1,2$. The two servers have exponential service times with parameters $\mu_i$ and a customer, after being served has a probability $p_{i,0}$ of exiting the system and $p_{i,j}$ of being transferred to node $j$, for $j=1,2$. Therefore, for a jump of amplitude $\mathbf{j}=(j_1,j_2)$ we have the following rates:
\begin{equation*}
rate(\mathbf{j})\colon=\left\{
\begin{array}{ll}
\lambda_i & \text{ for } \mathbf{j}=\mathbf{e}_i,\\
\mu_i p_{i,0} & \text{ for } \mathbf{j}=-\mathbf{e}_i,\\
\mu_i p_{i,j} & \text{ for } \mathbf{j}=-\mathbf{e}_i+\mathbf{e}_j.\\
\end{array}\right.
\end{equation*}
The rates do not depend on the current value of the process $Z^2$, with the only exception that, if the $i$-th component is equal to zero, jumps that decrease that component are suppressed. \\

Because of the nature of the jumps of the gap process $G^2$, where the increase by one unit of a component causes the decrease by one unit of another, except that for the ``last gap'', we associate to the gap process $G^2$ a particular two dimensional Jackson network. Let $Z^2$ be such that its parameters take the following values:
\[
\begin{array}{llll}
\lambda_1=0, & \mu_1=1, & p_{1,0}=0, & p_{1,2}=1,\\
\lambda_2=1+\delta, & \mu_2=2+\frac{\lambda}{2}+\delta, & p_{2,0}=\frac{+\frac{\lambda}{2}}{\mu_2}, & p_{2,1}=\frac{1+\delta}{\mu_2}.\\
\end{array}\]
The process $Z^2$ defined in this way has the same jumps and the same rates of $G^2$ in the internal region $\mathbb{N}_*\times \mathbb{N}_*$, while has a slight difference in the rates on the boundaries, see Figure~\ref{fig:GAP_N2} and \ref{fig:JACKSON_N2}.

\begin{figure}
\centering
  \begin{minipage}[b]{6.5cm}
   
    \includegraphics[width=0.99\textwidth]{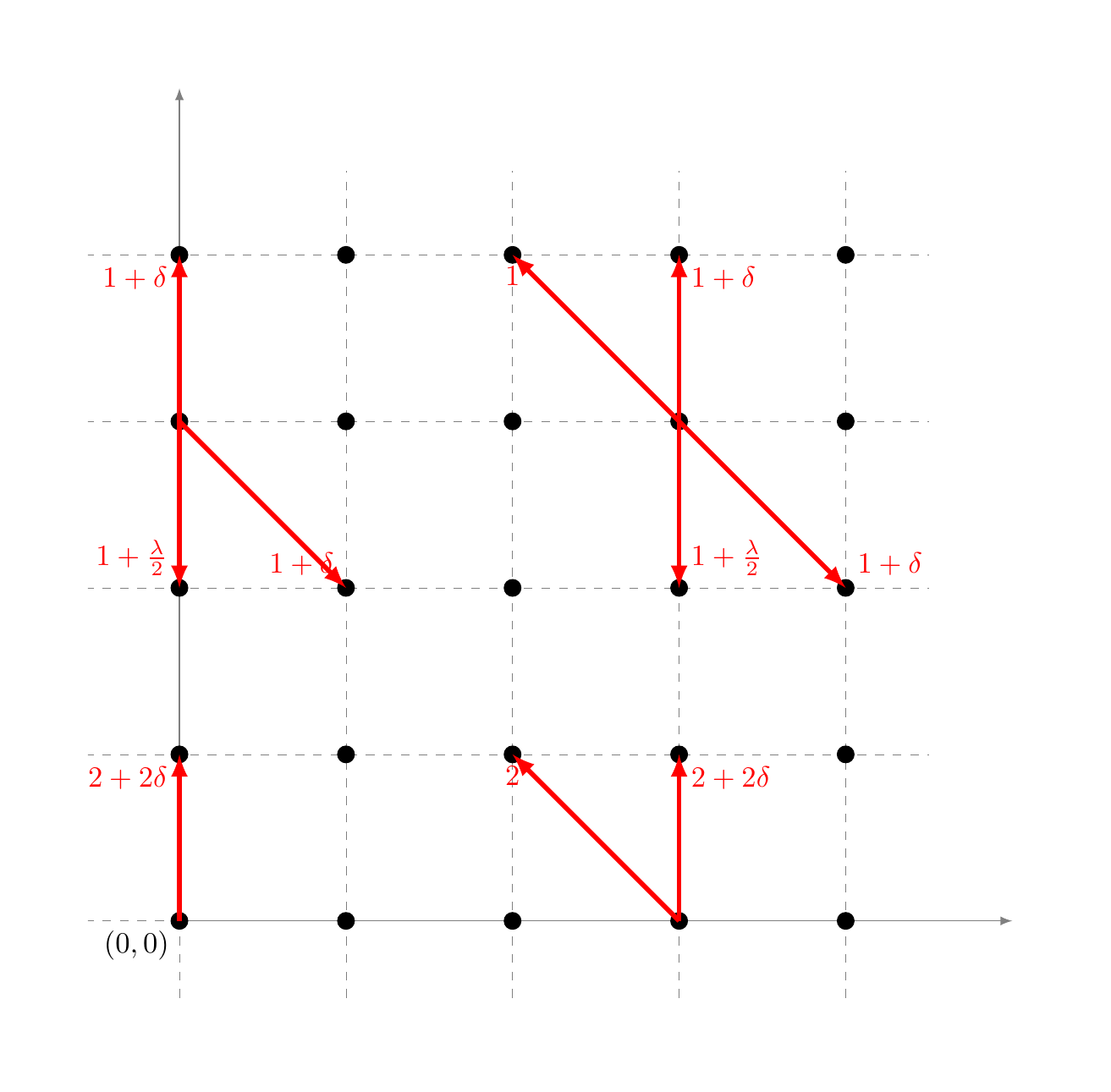}
    \caption{Jump rates of the gap process $G^2$. }
    \label{fig:GAP_N2}
  \end{minipage}
 \ \hspace{2mm} \hspace{3mm} \
 \begin{minipage}[b]{6.5cm}
     \includegraphics[width=\textwidth]{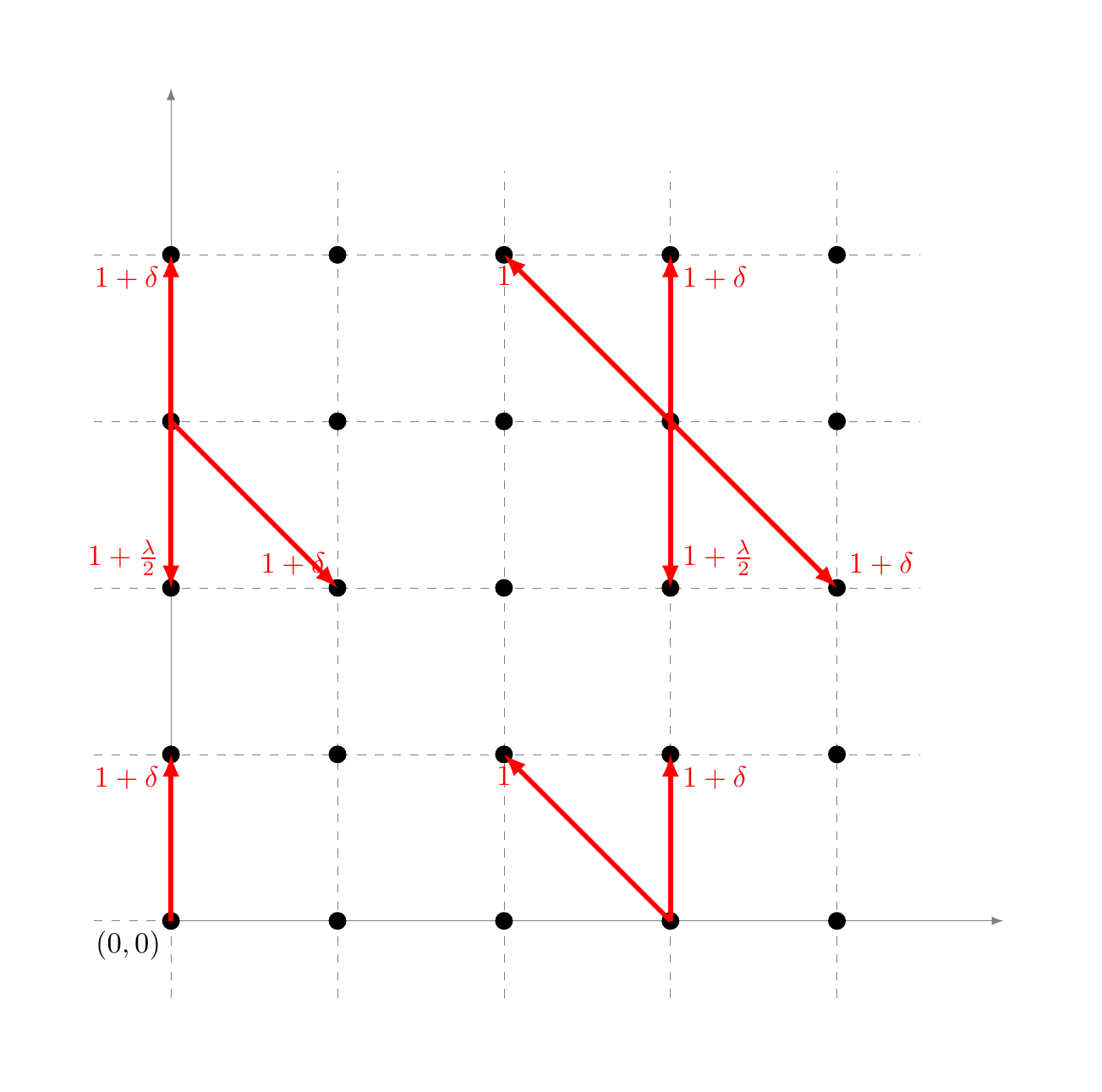}
    \caption{Jump rates of the Jackson network $Z^2$.}
    \label{fig:JACKSON_N2}
  \end{minipage}
\end{figure}
The two processes have embedded Markov chains with the same transition matrix. This implies that conditions for ergodicity are the same for both processes, with the same stationary measure.
 \begin{proof}[Proof of Theorem~\ref{Exact_N=2}]
 Consider the Jackson network $Z^2$ with same rates of $G^2$ in the internal region. Let $(\nu_1,\nu_2)$ be the solution of the so-called \emph{Jackson's system}:
 \[
 \left\{\begin{array}{l}
 \nu_1=\lambda_1+\nu_2p_{2,1},\\
 \nu_2=\lambda_2+\nu_1 p_{1,2}.
 \end{array}\right.\]
Classical results on Jackson networks, see \cite[Theorem~3.5.1]{FaMaMe95}, say that $Z^2$ is ergodic if and only if $\nu_i<\mu_i$, for i=1,2. In our case this condition becomes
\[\left\{\begin{array}{l}
\frac{(1+\delta)^2}{(1+\frac{\lambda}{2})}<1,\\
\frac{(1+\delta)\mu_2}{(1+\frac{\lambda}{2})}<\mu_2,
\end{array} \right.\]
that gives $\lambda>2\delta^2+4\delta$. In \cite{FaMaMe95}, by the use of a Lyapunov function, the authors prove that this is the necessary and sufficient condition for exponential ergodicity of the process $Z^2$ and, consequently, for $G^2$. The explicit form of $\pi_2$ comes from the adaptation of the product form stationary measure of a Jackson network and it is validated by verifying that $\pi_2$ solves the stationary equation, i.e. for all bounded measurable functions $f$ it holds:
\[
\sum_{(x,y)\in \mathbb{N}^2}\mathbf{L}^2f(x,y)\pi_2(x,y)=0,
\]
where $\mathbf{L}^2$ is the infinitesimal generator of $G^2$.
  \end{proof}
Theorem~\ref{Exact_N=2} gives the exact critical value $\lambda^*_2(\delta)$ for the ergodicity of the system and we see that it is quadratic in $\delta$. Moreover, the explicit expression of $\pi_2$ proves that, in the stationary regime, the gaps $G^2_1$ and $G^2_2$ are independent. Notice that the lower bound on $\lambda^*_2(\delta)$ obtained in Theorem~\ref{THM:lower_bound} is optimal in this case. 

\subsection{Conjectures on critical values for gap process when $N\geq3$}
The link between the gap process and a Jackson network for $N=2$ suggests a correspondence between gap processes
and Jackson network for any $N$. Unfortunately, when $N\geq3$ the transition matrix of the embedded Markov chains of $G^N$ and $Z^N$ are not the same. However we can propose a conjecture on the critical value $\lambda^*_N(\delta)$ based on the properties of the Jackson network. First of all, let us define the Jackson network $Z^N$ corresponding to the gap process $G^N$, for a fixed $N\geq 3$. $Z^N$ must be such that the transition rates in the internal region $\mathbb{N}_*^N$ correspond to the ones of the gap process $G^N$. For all $i=1,\dots, N-1$
\begin{equation}\label{rate_Z^N}
\begin{array}{lclcl}
z&\rightarrow & z-\mathbf{e}_i+\mathbf{e}_{i+1}\, &\, \text{ with rate } &\, 1+\lambda\frac{i-1}{N},\\
z&\rightarrow & z+\mathbf{e}_i-\mathbf{e}_{i+1}\, &\, \text{ '' }&\, 1+\delta,\\
z&\rightarrow &z-\mathbf{e}_N\, &\, \text{ '' }&\, 1+\lambda\frac{N-1}{N},\\
z&\rightarrow & z+\mathbf{e}_N\, &\, \text{ '' }&\, 1+\delta,
\end{array}
\end{equation}
where $\mathbf{e}_i$ is the vector $(0,\dots,0,1,0,\dots,0)$ with the $i$-th coordinate equal to $1$.
\begin{proposition}\label{prop_associated_jacksonN}
Fix $N\geq3$, the $N$ node Jackson network $Z^N$ with transition rates \eqref{rate_Z^N} is ergodic if, and only if,  condition \eqref{condition_N_Y^N} holds.
%\begin{equation}
%\label{condition_N}
%\displaystyle{\frac{(1+\delta)^N}{\displaystyle{\prod_{k=1}^{N-1}(1+\lambda\frac{k}{N})}}<1.} 
%\end{equation}
\end{proposition}
\begin{proof}
The Jackson network $Z^N$ is such that 
\[\begin{array}{l}
\lambda_N=1+\delta,\\
\lambda_j=0\text{ for all }j=1,\dots,N-1,\\
\\
\mu_1=0, \\
\mu_j=2+\delta+\lambda\frac{j-1}{N},\text{ for all }j=2,\dots,N,\\
\\
p_{1,2}=1, 
p_{1,k}=0, \text{ for all }k\neq2,\\
\\
p_{j,j+1}=\frac{1+\lambda\frac{j-1}{N}}{\mu_j},  \, \, p_{j,j-1}=\frac{1+\delta}{\mu_j} \text{ for all }j=2,\dots,N-1,\\
p_{j,k}=0 \text{ for all }j=2,\dots,N-1, \, \text{ and all } k\neq j+1,j-1,\\
\\
p_{N,0}=\frac{1+\lambda\frac{N-1}{N}}{\mu_N},  \, \, p_{N,N-1}=\frac{1+\delta}{\mu_j} \\
p_{N,k}=0 \text{ for all } k\neq N,0.\\
\end{array}\]
Let us recall the \emph{Jackson system}:
\[
\nu_j=\lambda_j+\sum_{i=1}^N\nu_ip_{i,j}, \, \, \text{ for } j=1,\dots,N.\]
It is easy to verify that the solution $(\nu_1,\dots,\nu_N)$ of this is system has the following form:
\[
\nu_j=\mu_j\displaystyle{\prod_{k=1}^{N+1-j}\frac{(1+\delta)}{(1+\lambda\frac{N-k}{N})}, \,\,\, \text{ for all } i=1,\dots,N,}\]
that by classical result on Jackson networks gives the following condition: \begin{equation*}
\displaystyle{\prod_{k=1}^i\frac{(1+\delta)}{(1+\lambda\frac{N-k}{N})}<1, \,\,\, \text{ for all } i=1,\dots,N,}
\end{equation*}
that is equivalent to \eqref{condition_N_Y^N}.
\end{proof}
The association of each gap process $G^N$ with the correspondent Jackson network $Z^N$ justifies Conjecture~\ref{conj:N}. Indeed, it would give an exact critical value $\lambda^*_N(\delta)$, i.e. for each $N\geq3$ and each $\delta\geq0$ would be the solution of  
(\ref{condition_N_Y^N}). In the continuous framework, the sequence of critical values (that by abuse of notation we indicate in the same way) $\lambda^*_N(\delta)$ converges, as $N$ goes to $\infty$ to the critical value $\lambda^*_{\infty}(\delta)$ for the nonlinear process. In our case we could not understand if this can be true or not, since we do not even know if there is a value such that there exists a unique stationary measure. However we derive from this Conjecture~\ref{conj:lim} on the critical value for which there exists at least one stationary measure based on the sequence $\lambda^*_{N}(\delta)$.\\
%
%
%
%==========================================================
%==========================================================
%==========================================================
%
%
%
%
\noindent{\bf Acknowledgements:} A.A. would like to thank Djalil Chafai and Denis Villemonais for discussions on the model.

\bibliographystyle{abbrv} 
\bibliography{Bibliography}
\end{document}